\documentclass[11pt,english]{amsart}

\def\margin_comment#1{\marginpar{\sffamily{\tiny #1\par}\normalfont}}
\usepackage[table]{xcolor}
\usepackage{amssymb,euscript}
\usepackage{amsfonts}
\usepackage{geometry}

\usepackage{graphicx}
\usepackage{setspace}
\usepackage{color}
\usepackage{amsmath,amssymb,amsthm}
\usepackage{epsfig}
\usepackage{babel}

 \usepackage{amsmath}

\usepackage{lipsum}
\usepackage{mathtools}
\usepackage{mathrsfs}
\usepackage{graphicx}
\usepackage{graphics}
\usepackage{setspace}
	\usepackage{color}

\usepackage{enumerate}
\usepackage{float}
\usepackage[utf8]{inputenc}
\usepackage[T1]{fontenc}
\usepackage[document]{ragged2e}
 \usepackage{verbatim}

\usepackage{arydshln}

\newcommand*{\mat}{\mathbf}

\makeatletter

\newtheorem{thm}{Theorem}[section]
\numberwithin{equation}{section} 
\numberwithin{figure}{section} 
\theoremstyle{plain}
\newtheorem*{thm*}{Theorem}
\theoremstyle{definition}
\theoremstyle{plain}
\newtheorem{thm_A}{Theorem}
\newtheorem*{defn*}{Definition}

\theoremstyle{plain}

\theoremstyle{plain} 

\theoremstyle{plain}

\newtheorem{prop}[thm]{Proposition} 
\theoremstyle{remark}
\newtheorem{ex}[thm]{Example}
\theoremstyle{remark}
\newtheorem{rem}[thm]{Remark}
\theoremstyle{plain}

\theoremstyle{plain}

\theoremstyle{plain}
\newtheorem{lem}[thm]{Lemma} 
\theoremstyle{definition}
\newtheorem{defn}[thm]{Definition}

\newtheorem*{acknowledgment*}{Acknowledgment}

\theoremstyle{plain}
\newtheorem*{ex*}{Example}
\theoremstyle{plain}

\begin{document}
\title[When the Tracy-Singh product of matrices represents an  operation ]{When the Tracy-Singh product of matrices represents a certain operation on linear operators}
\author{Fabienne Chouraqui}
\begin{abstract}
	Given two  linear transformations, with representing matrices $A$ and $B$ with respect to some bases, it is not clear, in general, whether  the Tracy-Singh product of  the matrices $A$ and $B$ corresponds to a particular operation  on the linear  transformations.  Nevertheless, it is not hard to show that in the  particular case that each matrix  is  a square  matrix of order of the form $n^2$, $n>1$,  and is partitioned into $n^2$  square blocks of order $n$, then their  Tracy-Singh product,  $A \boxtimes B$,   is similar to  $A \otimes B$,  and the change of basis matrix is a permutation matrix. 
In this note, we prove that in the special case of  linear operators  induced from set-theoretic solutions  of the Yang-Baxter equation, the  Tracy-Singh product  of their representing matrices is the representing matrix of the linear operator  obtained from the direct product of the set-theoretic solutions.
\end{abstract}

\maketitle
 Keywords: Tracy-Singh product of matrices, the Yang-Baxter equation, Set-theoretic solutions of the  Yang-Baxter equation, representing matrices of linear operators.
\section*{Introduction}
The Kronecker product  (or tensor product) of matrices is a fundamental concept in linear algebra, which arises in various areas of mathematics, physics, and engineering. Given two arbitrary matrices $A$ and $B$, of order $m\times n$ and $p \times q$ respectively, their Kronecker product, $A \otimes B\,=\,\begin{pmatrix}
a_{11}B &a_{12}B&...&a_{1n}B\\
...\\
a_{m1}B &a_{m2}B&...&a_{mn}B\\
\end{pmatrix}$,   a block  matrix of order $mp \times nq$. The Kronecker product  is associative and distributive (over $+$), but not commutative. If $A$ and $B$ are both square and invertible matrices, then $A \otimes B$ is also invertible and $(A \otimes B)^{-1}=A^{-1}\otimes B^{-1}$. There are several other properties of the Kronecker product which are well-known. The Tracy-Singh product of matrices is  a generalisation of the Kronecker product of matrices, called sometimes the block Kronecker product, as it requires a partition of the matrices $A$ and $B$ into blocks.  In case $A$ and $B$ are not  partitioned into blocks, their Tracy-Singh product, $A \boxtimes B$, is equal to $A \otimes B$. The Tracy-Singh product of matrices was introduced by Tracy and Singh in 1972 and has since gained significant attention in both mathematics and physics. The  Tracy-Singh product shares several properties with the Kronecker product. Like the Kronecker product, the Tracy-Singh product  is associative and distributive, but not commutative. If $A$ and $B$ are both square and invertible matrices, then $A \boxtimes B$ is also invertible and $(A \boxtimes B)^{-1}=A^{-1}\boxtimes B^{-1}$.

Despite the recent surge in interest in the Tracy-Singh product,  it is still not much understood. Moreover, a general  interpretation of the Tracy-Singh product is not known. In this note, we give an  interpretation  of the Tracy-Singh product  for a particular class of   linear operators. Indeed, in the special case of  linear operators  induced from set-theoretic solutions  of the Yang-Baxter equation, the  Tracy-Singh product  of their representing matrices is the representing matrix of the linear operator  obtained from the direct product of the set-theoretic solutions.

The Yang-Baxter equation (YBE) is an equation in mathematical physics and it lies in the  foundation of  the theory of quantum groups. One of the fundamental problems is to find all the solutions of this equation. In \cite{drinf}, Drinfeld suggested the study of a particular class of solutions,  derived from the so-called set-theoretic solutions.  A  set-theoretic solution of the Yang-Baxter equation or st-YBE is a pair $(X,r)$, where $X$ is a set and 
\[r: X \times X \rightarrow X \times X\,,\;\;\; r(x,y)=(\sigma_x(y),\gamma_y(x))\]
is a bijective map satisfying $r^{12}r^{23}r^{12}=r^{23}r^{12}r^{23}$, where $r^{12}=r \times Id_X$ and  $r^{23}=Id_X\times r$. A st-YBE  $(X,r)$ is said to be non-degenerate if, for every $x \in X$, the maps $\sigma_x,\gamma_x$ are bijections of $X$ and it  is said to be involutive if $r^2=Id_{X\times X}$.   Non-degenerate and involutive st-YBE give rise to solutions of the  Yang Baxter equation.  Indeed, by defining  $V$ to be  the  real vector space spanned by  a set in bijection with $X$,  and $c:V \otimes V \rightarrow V \otimes V$  to be the  linear operator induced by  $r $, then   $c$ is  a linear operator satisfying the equality  $c^{12}c^{23}c^{12}=c^{23}c^{12}c^{23}$ in $V \otimes V \otimes V$, that is $c$  is a solution of the  Yang-Baxter equation, called   the linearisation of $r$.  By composing $r$ with $\tau$, 
where $\tau$ is the flip map $\tau(x,y)=(y,x)$,  the induced linear operator, $R$,  is  a linear operator satisfying the equality  $R^{12}R^{13}R^{23}=R^{23}R^{13}R^{12}$ in $V \otimes V \otimes V$, that is $R$  is a solution of the quantum Yang-Baxter equation. We are now able to state:

\begin{thm_A}\label{thmA}
	Let $c:V\otimes V\rightarrow V \otimes V$ and $d:V'\otimes V'\rightarrow V' \otimes V'$ be  the linearisations of    the non-degenerate and involutive st-YBE $(X,r)$ and $(Y,s)$  respectively.  	Let  $e:W\otimes W\rightarrow W \otimes W$  be  the linearisation of   $(Z,t)$,   the direct product of  $(X,r)$ and $(Y,s)$, with $Z =X \times Y$. 
	Then $c \boxtimes d$,  the Tracy-Singh product of  $c$ and $d$,   is the representing matrix of  the linear automorphism $e$  with respect to the basis $Z\otimes Z$.
	\end{thm_A}
The paper is organized as follows. In Section $1$, we give some preliminaries on matrix products and their properties.  In Section $2$, we give some preliminaries on  solutions of the Yang-Baxter equation and on st-YBE.   In Section $3$,   we   prove Theorem \ref{thmA}. 
 \begin{acknowledgment*}
 	The author is very grateful to Gigel Militaru for  some instructive remarks on the direct product of  st-YBE.
 \end{acknowledgment*}
 
 The author confirms that the data supporting the findings of this study are available within the article.
 \section{Preliminaries on some products of matrices and their properties} 
 We refer to \cite{hyland},  \cite{liu}, \cite{commut}, \cite{tracy}, \cite{tracy-jina}, \cite{zhour}  for more details. 	We also refer to  Section $2$  in \cite{chou-tracy} for more preliminaries and examples.
 
 \begin{defn}
 	Let $A=(a_{ij})$ be  of order $m \times n$ and $B=(b_{kl})$ of order $p \times q$. Let $A=(A_{ij})$ be partitioned with $A_{ij}$ of order $m_i \times n_j$ as the $ij$-th block submatrix and let $B=(B_{kl})$ be partitioned with $B_{kl}$ of order $p_k \times q_l$ as the $kl$-th block submatrix ($\sum m_i=m,\,\sum n_j=n,\,\sum p_k=p,\,\sum q_l=q$).
 \begin{enumerate}
 	\item  The Kronecker (or tensor) product:
 	\begin{equation*}
 	A\,\otimes \, B\,=\,(a_{ij}B)_{ij}
 	\end{equation*}
 	The matrix $A\,\otimes \, B$ is of order $mp \times nq$ and the block $a_{ij}B$ is order $p \times q$.
 	\item  The Tracy-Singh (or block Kronecker)  product:
 	\begin{equation*}
 	A\,\boxtimes \, B\,=\,((A_{ij} \otimes B_{kl})_{kl})_{ij}
 	\end{equation*}
 	The matrix $A\,\boxtimes \, B$ is of order $mp \times nq$ and the block $A_{ij} \otimes B_{kl}$ is order $m_ip \times n_jq$.
 \end{enumerate}
 For non-partitioned matrices,  $A\,\boxtimes \, B\,=\, A\,\otimes \, B$.
\end{defn}
\begin{ex}\label{ex-tracyproduct-2-3}
	The Tracy-Singh product of $A$ and $B$ with those partitions is:
\begin{gather} 
\nonumber A \boxtimes B=\left( \begin{tabular}{c|c}
$A_{11}$  &	 	$A_{12}$  \\
\hline
$A_{21}$ & $A_{22}$ \\
\end{tabular}\right) \boxtimes
\left( \begin{tabular}{c|c|c}
$B_{11}$ &	$B_{12}$  &	$B_{13}$\\
\hline
$B_{21}$ &	$B_{22}$  &	$B_{23}$\\
\hline
$B_{31}$ &	$B_{32}$  &	$B_{33}$\\
\end{tabular}\right)=\\
\left( \begin{tabular}{c|c|c||c|c|c}\label{eqn-tracy-blocks}
$A_{11}\otimes B_{11}$  &	 $A_{11} \otimes B_{12}$  &	$A_{11}\otimes B_{13}$ &	$A_{12}\otimes B_{11}$  &	$A_{12}\otimes B_{12}$ &	$A_{12}\otimes B_{13}$\\
\hline
$A_{11}\otimes B_{21} $&	$A_{11}\otimes B_{22}$ &	$A_{11}\otimes B_{23}$ &	$A_{12}\otimes B_{21}$ &	$A_{12}\otimes B_{22}$ &	$A_{12}\otimes B_{23}$ \\
\hline
$A_{11}\otimes B_{31}$ &	$A_{11}\otimes B_{32}$ &	$A_{11}\otimes B_{33} $&	$A_{12}\otimes B_{31} $&	$A_{12}\otimes B_{32}$ &	$A_{12}\otimes B_{33}$ \\
\hline
\hline
$A_{21}\otimes B_{11}$ &	$A_{21}\otimes B_{12}$ &	$A_{21}\otimes B_{13} $&	$A_{22}\otimes B_{11}$ &	$A_{22}\otimes B_{12} $&	$A_{22}\otimes B_{13}$ \\
\hline
$A_{21}\otimes B_{21}$ &	$A_{21}\otimes B_{22}$ &	$A_{21}\otimes B_{23}$ &	$A_{22}\otimes B_{21} $&	$A_{22}\otimes B_{22}$ &	$A_{22}\otimes B_{23}$ \\
\hline
$A_{21}\otimes B_{31}$ &	$A_{21}\otimes B_{32} $&	$A_{21}\otimes B_{33} $&	$A_{22}\otimes B_{31}$ &	$A_{22}\otimes B_{32} $&	$A_{22}\otimes B_{33}$ \\
\end{tabular}\right)
\end{gather} 
\end{ex}
 In the following Theorems, we list some important properties of the Tracy-Singh product.
 \begin{thm}\cite{tracy}\label{thm-tracy}
 	Let $A$, $B$, $C$, and $D$  be matrices. Then
 	\begin{enumerate}[(i)]
 		\item $A \boxtimes B$ and $B \boxtimes A$ exist for any  matrices $A$ and $B$.
 		\item $A \boxtimes B \neq B \boxtimes A$ in general.
 		\item  $(A \boxtimes B)\,\boxtimes C= \,A\boxtimes \,(B \boxtimes C)$.
 		\item $(A+ B)\, \boxtimes\,(C +D)=\, A \boxtimes C+A \boxtimes D +B \boxtimes C+B \boxtimes D$,  if $A+B$ and $C+D$ exist.
 		\item $(A \boxtimes B)\,(C \boxtimes D)=\, AC \boxtimes BD$,  if $AC$ and $BD$ exist.
 		\item $(cA) \boxtimes B\,=\,c (A\boxtimes B \,=\,A \boxtimes (cB)$.
 		\item  $(A \boxtimes B)^{-1}\,=  A^{-1} \boxtimes B^{-1}$, if $A$ and $B$ are invertible.
 		\item  $(A \boxtimes B)^{t}\,=  A^{t} \boxtimes B^{t}$.
 		\item 	$\mat{I}_{n} \boxtimes \mat{I}_{m}\,=\,\mat{I}_{nm}$ for identity partitioned matrices.
 	\end{enumerate}
 \end{thm}
In matrix theory, the commutation matrix is used for transforming the vectorized form of a matrix into the vectorized form of its transpose. 
\begin{defn}\cite{commut} \label{defn-commut}
The \emph{commutation matrix} $K_{mn}$  is the  matrix defined by:
\begin{equation*}
K_{mn}=\, \sum\limits_{i=1}^{i=m}\,\sum\limits_{j=1}^{j=n}E_{ij}\otimes E^t_{ij}
\end{equation*}
where $E_{ij}$ is a matrix of order $ m \times n$ with a $1$ in its $ij$-th position and zeroes elsewhere.\\
In words, $K_{mn}$  is the square matrix of order $mn$, partitioned into $ mn$ blocks of order $ n \times m$ such that the $ij$-th block has a 1 in its $ ji$-th position and $0$  elsewhere.  It holds that $K_{nm}\,=\,	K_{mn}^{-1}$.
\end{defn}
For example (see \cite[p383]{commut}), $K_{23}=
\left(\begin{array}{c:c|c}
E_{11}
& E_{21}
& E_{31} \\
\hdashline
E_{12}
& E_{22}
& E_{32} \\
\end{array}\right)\;\;=\;\;
\left(\begin{tabular}{cc:cc:cc}
1

& 0 & 0
& 0 &0&0 \\
0

& 0 & 1
& 0  & 0&0 \\
0

& 0  & 0
& 0 & 1&0 \\
\hdashline
0

& 1 & 0
& 0  & 0&0 \\

0

& 0 & 0
& 1  & 0&0 \\
0

& 0  & 0
& 0 & 0&1 \\
\end{tabular}\right)$

\begin{prop}\label{prop-prop-box-cd}
Let $A$ and $B$ be square matrices of order $n^2$ and $m^2$  respectively,  such that $A$  has a partition into $n^2$ square blocks $B_{ij}$ of order $n$ and $B$ has  a    partition into $m^2$ square blocks $B'_{kl}$ of order $m$.  Let $A \boxtimes B$ denote the Tracy-Singh product of $A$ and $B$ with these partitions. Then
\begin{enumerate}[(i)]
\item $B \otimes A\,=\, K_{m^2n^2}\,\cdot\,(A \otimes B)\,\cdot\,K_{n^2m^2}$, where $K_{m^2n^2}\,=\,(K_{n^2m^2})^{-1}$.
\item 	
$A \boxtimes B\,=\, 	(\mat{I}_{n}\,\otimes K_{mn} \,\otimes \mat{I}_{m})\,	\cdot\, (A\otimes B)\, \cdot\, (\mat{I}_{n}\,\otimes K_{nm}\,\otimes  \mat{I}_{m})$

\item 
$B\boxtimes A\,=\, 	P \cdot\, (A\boxtimes B)\,\cdot\,P^{-1}$, where 
$P=\, (\mat{I}_{m}\,\otimes K_{nm} \,\otimes \mat{I}_{n})\,	\cdot\, K_{m^2n^2}\,\cdot\, (\mat{I}_{n}\,\otimes K_{nm} \,\otimes \mat{I}_{m})$
\end{enumerate}
\end{prop}
Note that, 	from the properties of the Kronecker product and the commutation matrix, $ (\mat{I}_{n}\,\otimes K_{mn} \,\otimes \mat{I}_{m})^{-1}\,=\, (\mat{I}_{n}\,\otimes K_{nm}\,\otimes  \mat{I}_{m})$.
For more general formulas, we refer to \cite{neud},\cite{tracy-jina}, \cite{commut},  and also \cite{chou-tracy}.
 \section{Preliminaries on  solutions of the   Yang-Baxter equation  (YBE) }
 \subsection{Definition and properties of solutions of the YBE}
 We use the terminology from \cite[Ch.VIII]{kassel}.  We refer the reader to  \cite{kassel} for more details.
 \begin{defn}
 	Let $V$ be a vector space over a field $k$. A linear automorphism $c$ of  $V \otimes V$ is said to be an $R$-matrix if it is a solution of the Yang-Baxter equation (YBE)
 	\begin{equation}\label{eqn-ybe}
 	(c \otimes Id_V)(Id_V \otimes c )(c \otimes Id_V)\,=\,(Id_V \otimes c)(c \otimes Id_V)(Id_V\otimes c)
 	\end{equation}
 	that holds in the automorphism group of  $V \otimes V \otimes V$.
 	It is also written as $c^{12}c^{23}c^{12}=c^{23}c^{12}c^{23}$.
 \end{defn}
 Let $\{e_i\}_{i=1}^{i=n}$ be a basis of the finite dimensional vector space $V$.  An automorphism $c$ of   $V \otimes V$ is defined by the family $(c_{ij}^{kl})_{i,j,k,l}$ of scalars determined by
 
 \begin{equation}\label{eqn-defn-c}
 c(e_i\otimes e_j)\,=\,\sum\limits_{k,l} \,c_{ij}^{kl}\,e_k \otimes e_l
 \end{equation}
 Then $c$ is an $R$-matrix if and only if  its representing  matrix  with respect to  the basis  $\{e_i\otimes e_j\mid 1\leq i,j\leq n\}$,   which is also denoted by $c$, satisfies
 \begin{equation}\label{eqn-ybe-matrix}
 (c \otimes  \mat{I}_{n})\,(\mat{I}_{n}\otimes  c)\,(c \otimes  \mat{I}_{n})\,=\,(\mat{I}_{n}\otimes  c)\,(c \otimes  \mat{I}_{n})\,(\mat{I}_{n}\otimes  c)
 \end{equation}
 \begin{ex}\label{ex-solutions}(\cite{kaufman}, \cite{dye} and \cite{kassel})
 	Let $c,\,d:V \otimes V \rightarrow V \otimes V$ be  $R$-matrices,  $\operatorname{dim}(V)=2$:
 	\[c=	\begin{pmatrix}
 	\frac{1}{\sqrt{2}}& 0 & 0 & 	\frac{1}{\sqrt{2}}\\
 	0 &	\frac{1}{\sqrt{2}} & -	\frac{1}{\sqrt{2}} & 0 \\
 	0 &	\frac{1}{\sqrt{2}} & 	\frac{1}{\sqrt{2}} &0 \\
 	-	\frac{1}{\sqrt{2}} &0 & 0 &	\frac{1}{\sqrt{2}} \\
 	\end{pmatrix}\;\;\;  \textrm{and}\;\;\; 
 	d=\begin{pmatrix}
 	2& 0 & 0 & 	0\\
 	0 &0 & 1 & 0 \\
 	0 &	1 & 1.5&0 \\
 	0&0 & 0 &	2\\
 	\end{pmatrix}\]
 	From the matrices $c$ and $d$, we read  $c(e_1 \otimes e_2)=\,\frac{1}{\sqrt{2}} \,e_1 \otimes e_2\,+\,\frac{1}{\sqrt{2}} \,e_2 \otimes e_1$ and $d(e_1 \otimes e_2)=\,e_2 \otimes e_1$.
 		As a convention, we  always consider the  basis   $\{e_i \otimes e_j \,\mid \,1\leq i,j \leq  n \}$ of $V\otimes V$ ordered lexicographically, that is, as an example, for $n=2$, the ordered basis of $V\otimes V$  is 
 		$\{e_1\otimes e_1 ,\,e_1\otimes e_2,\,e_2\otimes e_1,\,e_2\otimes e_2 \}$.
 	\end{ex}
 
 \begin{defn}\label{def-iso}
 	Let $c: V \otimes V \rightarrow V \otimes V$ and $c': V' \otimes V' \rightarrow V' \otimes V'$
 	be solutions of the YBE, where $V$ and $V'$ are vector spaces over the same field $k$, with $\operatorname{dim}(V)=\operatorname{dim}(V')=n$. The solutions $c$ and  $c'$ are \emph{isomorphic} if there exists a linear isomorphism $\mu:V \rightarrow V'$ such that 
 	$c'\,(\mu \otimes\mu)\,=\,(\mu \otimes\mu)\,c$.
 \end{defn}
 There  is a  way to generate new $R$-matrices from old ones. Indeed,  if $c \in \operatorname{Aut}(V \otimes V)$ is an $R$-matrix, then so are $\lambda c$, $c^{-1}$ and $\tau \circ c \circ \tau$, where $\lambda$ is any non-zero scalar and $\tau: V \otimes V$ is the flip map ($\tau(v_1 \otimes v_2)=v_2\otimes v_1$) \cite[p.168]{kassel}. 
 \begin{rem}\label{rem-similar-matrices}
 	If  $c \in \operatorname{Aut}(V \otimes V)$ is an $R$-matrix,  then $P^{-1}cP$ is not necessarily an  $R$-matrix,  where $P$ is any invertible matrix.  As an example, if  we  conjugate the $R$-matrix  $c$ (or $d$) from Example \ref{ex-solutions} by  $P$, where $P$ is  the permutation matrix  corresponding to the cycle  $(1,2,3,4)$, then the  matrix  $P^{-1}cP$  (or $P^{-1}dP$) is not  an  $R$-matrix. But,  if  we  conjugate the $R$-matrix  $c$ (or $d$) from Example \ref{ex-solutions} by  $P$, where $P$ is  the permutation matrix  corresponding to $(1,4)(2,3)$, then the  matrix  $P^{-1}cP$  (or $P^{-1}dP$) is   an  $R$-matrix. The explanation  is the following: in both  cases, $P^{-1}cP$ and $c$ are similar matrices and $P$ is the matrix  describing a change of basis  of  $V \otimes V$, but   only in the  second  case, the change of basis of  $V \otimes V$  corresponds to a change of basis of $V$ ( the new basis  of $V$ is  $\{v_1=e_2\,, v_2=e_1\}$). 
 \end{rem}
 
 \subsection{Definition and properties of set-theoretic solutions of the YBE}
 In \cite{drinf}, Drinfeld suggested the study of a particular class of solutions,  derived from the so-called set-theoretic solutions. The study of these solutions was initiated in \cite{etingof}. We refer also to  \cite{gateva_van,gateva_new}  and \cite{jespers_book}. Let $X$ be a non-empty set. Let $r: X \times X \rightarrow X \times X$  be a map and write $r(x,y)=(\sigma_{x}(y),\gamma_{y}(x))$,  where $\sigma_x, \gamma_x:X\to X$ are functions  for all  $x,y \in X$.   The pair $(X,r)$ is  \emph{braided} if $r^{12}r^{23}r^{12}=r^{23}r^{12}r^{23}$, where the map $r^{ii+1}$ means $r$ acting on the $i$-th and $(i+1)$-th components of $X^3$.  In this case, we  call  $(X,r)$  \emph{a set-theoretic solution of the  Yang-Baxter equation or st-YBE}, and whenever $X$ is finite, we  call  $(X,r)$  \emph{a finite st-YBE}.  The pair $(X,r)$ is \emph{non-degenerate} if for every  $x\in X$,  $\sigma_{x}$ and $\gamma_{x}$  are bijective and it   is  \emph{involutive} if $r\circ r = Id_{X^2}$. If $(X,r)$ is a non-degenerate involutive st-YBE, then $r(x,y)$ can be described as  $r(x,y)=(\sigma_{x}(y),\gamma_{y}(x))=(\sigma_{x}(y),\,\sigma^{-1}_{\sigma_{x}(y)}(x))$.  A st-YBE  $(X,r)$ is \emph{square-free},  if  for every $x \in X$, $r(x,x)=(x,x)$. A st-YBE  $(X,r)$ is \emph{trivial} if $\sigma_{x}=\gamma_{x}=Id_X$, for every  $x \in X$. 
 \begin{lem}\cite{etingof} \label{lem-formules-invol+braided}
 	\begin{enumerate}[(i)]
 		\item  $(X,r)$ is involutive  if  and only if  for every  $x,y \in X$:
 		\begin{gather}
 		\sigma_{\sigma_x(y)}\gamma_{y}(x)=x \label{eqn-inv1}\\
 		\gamma_{\gamma_y(x)}\sigma_x(y)=y \label{eqn-inv2}
 		\end{gather} 
 		\item   $(X,r)$ is  braided if  and only if, 	for  every $x,y,z \in X$, the following  holds:
 		\begin{gather}
 		\sigma_x\sigma_y=\sigma_{\sigma_x(y)}\sigma_{\gamma_y(x)} \label{eqn-braided-sigma}\\  \gamma_y\gamma_x=\gamma_{\gamma_y(x)}\gamma_{\sigma_x(y)} \label{eqn-braided-gamma}\\\gamma_{\sigma_{\gamma_y(x)}(z)}(\sigma_x(y))=\sigma_{\gamma_{\sigma_y(z)}(x)}(\gamma_z(y)) \label{eqn-braided-old}
 		\end{gather}
 	\end{enumerate}
 \end{lem}
 Non-degenerate and involutive st-YBE are  intensively investigated and they give rise to several algebraic structures associated to them   \cite{gigel,catino4,cedo,brace,chou_art,chou_godel2,doikou,lebed1,adolfo,rump} and  others.
 \begin{ex}\label{ex-sol-2}
 	For $\mid X\mid =2$,  there are exactly two non-degenerate and involutive st-YBE. The first one, $(X,r)$,  is called  a trivial solution with $\sigma_1=\sigma_2=\gamma_1=\gamma_2=Id_X$  and the second one, $(Y,s)$,  is called a permutation solution with  $\alpha_1=\alpha_2=\beta_1=\beta_2=(1,2)$.  Their matrices with respect to the ordered standard basis  $\{e_i \otimes e_j \mid 1 \leq i,j \leq 2\}$ are respectively 	$c=\begin{pmatrix}
 	1& 0 & 0 & 	0\\
 	0 &0 & 1 & 0 \\
 	0 &	1 & 0&0 \\
 	0&0 & 0 &	1\\
 	\end{pmatrix}$ and 	$d=\begin{pmatrix}
 	0& 0 & 0 & 	1\\
 	0 &1 & 0 & 0 \\
 	0 &	0 & 1&0 \\
 	1&0 & 0 &	0\\
 	\end{pmatrix}$.
 \end{ex}
 The following lemma is well-known and trivial. We add its short  proof.
 \begin{lem}\label{lem-i-sigma-1}
	Let $(X,r)$ be a non-degenerate involutive  st-YBE. Then 
	\begin{equation}\label{eqn-i-sigma-1}
	r(x_i,\,x_{\sigma_i^{-1}(j)})\,=\,(x_j,\,x_{\sigma_j^{-1}(i)})
	\end{equation}
\end{lem}
\begin{proof}
	From the definition of $r$,  $r(x_i,x_{\sigma_i^{-1}(j)})=(x_{\sigma_i\sigma_i^{-1}(j)}\,, \,x_{\gamma_{\sigma_i^{-1}(j)}(i)})=(x_j,\, x_{\gamma_{\sigma_i^{-1}(j)}(i)})$. By replacing $x$ by  $i$ and $y$ by  $\sigma_i^{-1}(j)$  in Equation (\ref{eqn-inv1}), we have  $\gamma_{\sigma_i^{-1}(j)}(i)=\sigma_j^{-1}(i)$, that is 	$r(x_i,\,x_{\sigma_i^{-1}(j)})\,=\,(x_j,\,x_{\sigma_j^{-1}(i)})$.
\end{proof}

\begin{defn}\label{defn-new-solution}
Let $(X,r)$ and $(Y,s)$  be non-degenerate and involutive  st-YBE,  where $X =\{x_1,..,x_n\}$, $Y=\{y_1,..,y_{n'}\}$,   $r(x_i,x_j)=(x_{\sigma_{i}(j)}\,,\,x_{\gamma_{j}(i)})$, $x_i,x_j \in X$ and $s(y_i,y_j)=(y_{\alpha_{i}(j)}\,,\,y_{\beta_{j}(i)})$, $y_i,y_j \in Y$. 
Let  $Z$ denote the set  $ X \times Y$. Let $t: Z^2 \rightarrow Z^2$ be  the following  induced map:
	\begin{gather}
	t: Z^2 \rightarrow Z^2\\
t(\,(x_i,y_k)\,,\,(x_j,y_l)\,)\,=\,(\;(x_{\sigma_i(j)},y_{\alpha_k(l)})\,,\,(x_{\gamma_j(i)},y_{\beta_l(k)})\;) \label{eqn-defn-new-r}
	\end{gather}
	The pair $(Z,t)$ is a st-YBE called \emph{the direct product of $(X,r)$ and $(Y,s)$}. 
\end{defn}
To make the paper self-contained, we prove in Lemma \ref{lem-new--invol+braided} that   $(Z,t)$ is indeed a st-YBE and that it  inherits the properties  of non-degeneracy and involutivity from  $(X,r)$ and $(Y,s)$.
 To shorten notation, we identify $Z$ with another set of elements  
 in bijection with $X \times Y$, and we  write:
 	\begin{gather}
 	Z=\, \{T_{i}^{k}\mid 1 \leq i \leq n,\, 1 \leq k \leq n'\} \label{defn-Z}\\
  t(T_{i}^{k},T_{j}^{l})\,=(g_i^k(T_{j}^{l})\,,\,f_j^l(T_{i}^{k}))\label{eqn-defn-t}
 \end{gather}
   where   $g_i^k$ and $f_j^l$, $1 \leq i,j \leq n$, $1 \leq k,l \leq n'$ are defined by:
	\begin{gather}
	g_i^k\,,\, f_j^l\,: Z\rightarrow Z\nonumber\\
	g_i^k(T_{j}^{l})\,=\, T_{\sigma_i(j)}^{\alpha_k(l)} \label{eqn-defn-new-g}\\
	f_j^l(T_{i}^{k})\,=\,T_{\gamma_j(i)}^{\beta_l(k)} \label{eqn-defn-new-f}
	\end{gather}
That is, in other words, $g_i^k(x_j,y_l)\,=\,(x_{\sigma_i(j)}\,,\, y_{\alpha_k(l)})$
 and $ f_j^l(x_i,y_k)\,=\,(x_{\gamma_j(i)}, y_{\beta_l(k)})$, or  $g_i^k\,=\, \sigma_i \times \alpha_k$ and $f_j^l\,=\,\gamma_j \times  \beta_l$.

 \begin{lem}\label{lem-new--invol+braided}
 Let $(X,r)$ and $(Y,s)$  be non-degenerate and involutive  st-YBE,  where $X =\{x_1,..,x_n\}$, $Y=\{y_1,..,y_{n'}\}$,   $r(x_i,x_j)=(\sigma_{i}(j),\gamma_{j}(i))$, $x_i,x_j \in X$ and $s(y_i,y_j)=(\alpha_{i}(j),\beta_{j}(i))$, $y_i,y_j \in Y$. 
 Let  $(Z,t)$ be their direct product, as defined 
  in Definition \ref{defn-new-solution}. Then 
 	\begin{enumerate}[(i)]
 	\item  $(Z,t)$ is non-degenerate, that is $g_i^k\,,\, f_j^l$, $1 \leq i,j \leq n\,,\,1 \leq k,l \leq n'$, are bijective.
 			\item   $(Z,t)$  is  involutive, that is  $t^2=\,Id_{Z^2}$.
 	\item    $(Z,t)$  is  braided, that is,  $t^{12}t^{23}t^{12}\,=\,t^{23}t^{12}t^{23}$. 	That is, for every  $1 \leq i,j,s \leq n$,  $1 \leq k,l,m \leq n'$ the following  equations hold:\\
 	 \begin{gather}
 	g_i^kg_j^l\,=\, g_{\sigma_i(j)}^{\alpha_k(l)}\,g_{\gamma_j(i)}^{\beta_l(k)} \label{eqn-braided-g}\\
 		f_j^lf_i^k\,=\, f_{\gamma_j(i)}^{\beta_l(k)} \,f_{\sigma_i(j)}^{\alpha_k(l)} \label{eqn-braided-f}\\
 			f_{\sigma_{\gamma_j(i)}(s)}^{\alpha_{\beta_l(k)}(m)}\,g_i^k(T_j^l)\;=\;
 		g_{\gamma_{\sigma_j(s)}(i)}^{\beta_{\alpha_l(m)}(k)}\,\,f_s^m(T_j^l)
 		\label{eqn-braided-new}
 	\end{gather} 
 	\end{enumerate}
 \end{lem}
\begin{proof}
	$(i)$  follows from the fact that  	$g_i^k(T_{j}^{l})\,=\, T_{\sigma_i(j)}^{\alpha_k(l)}$ and 
	$f_j^l(T_{i}^{k})\,=\,T_{\gamma_j(i)}^{\beta_l(k)}$, and  the functions   $\sigma_i$, $\alpha_k$ $\gamma_j$ and $\beta_l$ are bijective.\\
	$(ii)$ $t^2(T_{i}^{k},T_{j}^{l})\,= t(T_{\sigma_i(j)}^{\alpha_k(l)}\,,\,T_{\gamma_j(i)}^{\beta_l(k)}) \,=\, (T_{\sigma_{\sigma_i(j)}\gamma_j(i)}^{\alpha_{\alpha_k(l)}\beta_l(k)}\,,\, 
	T_{\gamma_{\gamma_j(i)}\sigma_i(j)}^{\beta_{\beta_l(k)}\alpha_k(l)})$. As $(X,r)$ and $(Y,s)$   are  involutive,  this is equal to $(T_{i}^{k},T_{j}^{l})$,  from Equations (\ref{eqn-inv1}), (\ref{eqn-inv2}). So,     $t^2=\,Id_{Z^2}$. 			\\
	$(iii)$  From the definition of $t$, 	$t^{12}t^{23}t^{12}\,=\,t^{23}t^{12}t^{23}$  if and only if  Equations (\ref{eqn-braided-g})-(\ref{eqn-braided-new}) hold. We prove (\ref{eqn-braided-g}). From  Equation (\ref{eqn-defn-new-g}), we have:
		\begin{align*}
	g_i^kg_j^l(T_s^m)\,=\,g_i^k(T_{\sigma_j(s)}^{\alpha_l(m)})=\,
	T_{\sigma_i\sigma_j(s)}^{\alpha_k\alpha_l(m)}\\
	g_{\sigma_i(j)}^{\alpha_k(l)}\,g_{\gamma_j(i)}^{\beta_l(k)}(T_s^m)=
	g_{\sigma_i(j)}^{\alpha_k(l)}(T_{\sigma_{\gamma_j(i)}(s)}^{\alpha_{\beta_l(k)}(m)})\,=\,T_{\sigma_{\sigma_i(j)}\sigma_{\gamma_j(i)}(s)}^{\alpha_{\alpha_k(l)}\alpha_{\beta_l(k)}(m)}
		\end{align*}
	From Equation (\ref{eqn-braided-sigma}), 
	$T_{\sigma_i\sigma_j(s)}^{\alpha_k\alpha_l(m)}\,=\,T_{\sigma_{\sigma_i(j)}\sigma_{\gamma_j(i)}(s)}^{\alpha_{\alpha_k(l)}\alpha_{\beta_l(k)}(m)}$,  for every $1 \leq s \leq n$, $1 \leq m \leq n'$,  so (\ref{eqn-braided-g}) holds. In the same way, we show (\ref{eqn-braided-f}) holds, using Equations   (\ref{eqn-defn-new-f}) and (\ref{eqn-braided-gamma}), and  (\ref{eqn-braided-new})
	using Equation (\ref{eqn-braided-old}).
\end{proof}
 \begin{ex}\label{ex-direct-product-sol-2}
	Let $(X,r)$  and  $(Y,s)$ be the solutions corresponding to $c$ and $d$ respectively in   Example   \ref{ex-sol-2}. Then, their direct product $(Z,t)$ is defined by $g_i^k=f_j^l=Id_{\{1,2\}}\times (1,2)$, for every $1 \leq  i,j,k,l\leq 2$, that is  $g_i^k=f_j^l$ are equal to the permutation $(T_1^1,T_1^2)(T_2^1,T_2^2)$. We have 
	$t(T_1^1,T_1^1)=(T_1^2,T_1^2)\,,\,
t(T_1^1,T_2^1)=(T_2^2,T_1^2)\,,\,
t(T_1^1,T_2^2)=(T_2^1,T_1^2)\,,\,
t(T_1^2,T_2^1)=(T_2^2,T_1^1)\,,\,
t(T_1^2,T_2^2)=(T_2^1,T_1^1)\,,\,
t(T_2^1,T_2^1)=(T_2^2,T_2^2)$.
\end{ex}

\section{Proof of Theorem \ref{thmA} }
In order to prove  Theorem \ref{thmA},  we describe in the following lemmas the blocks in the matrices  $c \boxtimes d$ and $e$, where $e$ is the representing matrix of the direct product of the  underlying  set-theoretic solutions of $c$ and $d$, and prove they are equal. Before we proceed with the technical computations, we present  the idea  of the proof of  the equality of the matrices  $e$ and $c \boxtimes d$.  First, we show  that in $e$ and also in $c \boxtimes d$, at each row, at each column and at each block, there is a unique non-zero entry $1$.  Next, we show that  for each block in   $e$ and also in $c \boxtimes d$,  the non-zero entry occurs at the same position in the block.  
\subsection{Properties of the  matrix $e$}
We first describe the properties of the representing matrix of an operator induced from a  st-YBE. Given a matrix $A$, we say that the scalar $a_{ij}$ is \emph{the entry at position $(i,j)$}.
\begin{lem}\label{lem-propblock--c}
		Let $(X,r)$  be  a  non-degenerate and involutive st-YBE, with $\mid X\mid=n$. Let $c$ be  the  linearisation of  $(X,r)$.  Assume that $c$  is partitioned into $n^2$ square blocks $B_{ij}$ of order $n$.  Then
		\begin{enumerate}[(i)]
			\item  At each row  and each column of $c$, there is a unique non-zero entry $1$.
			\item  At each block of $c$, there is a unique non-zero entry $1$.
			\item Let  $1 \leq i,j \leq n$. In the block $B_{ij}$: 
			\begin{itemize}
				\item we read the images of all the vectors $e_j\otimes e_{k}$, $1 \leq k \leq n$,  relative to the subset of basis elements $\{e_i\otimes e_{l}\mid 1 \leq l \leq n\}$ .
				\item the non-zero  entry occurs at position $(\sigma_i^{-1}(j)\,,\,\sigma_j^{-1}(i))$
			\end{itemize} 
	
		\end{enumerate}
\end{lem}

\begin{proof}
	$(i)$,  $(ii)$  $c$ is a permutation matrix  of order $n^2$,  so  there is a unique entry $1$  at each row  and each column of $c$ and there are   $n^2$ non-zero entries. Moreover, as there are  $n^2$ square blocks of order $n$, there is a unique non-zero entry at each block, since otherwise there would be a column and a row with more than one non-zero entry.\\
	$(iii)$  Let $c:V\otimes V\rightarrow V\otimes V$, with    respect to the basis   $\{e_i \otimes e_j \,\mid \,1\leq i,j \leq  n \}$ of $V\otimes V$   ordered lexicographically.  The first $n$ columns describe the images of  $\{e_1\otimes e_{k}\mid 1 \leq k \leq n\}$, the next  $n$ columns describe the images of  $\{e_2\otimes e_{k}\mid 1 \leq k \leq n\}$ and so on, the  last $n$ columns describe the images of  $\{e_n\otimes e_{k}\mid 1 \leq k \leq n\}$.   In the   first $n$ rows we read the images of  all the basis elements relative to the subset $\{e_1\otimes e_{l}\mid 1 \leq l \leq n\}$ and so on.  Let  $1 \leq i,j \leq n$. From the above, in the block $B_{ij}$, 	we read the images of all the vectors $e_j\otimes e_{k}$, $1 \leq k \leq n$,  relative to the subset of basis elements $\{e_i\otimes e_{l}\mid 1 \leq l \leq n\}$. From Lemma \ref{lem-i-sigma-1}, 
	$c(e_j\otimes e_{\sigma_j^{-1}(i)})\,=\,e_i\otimes e_{\sigma_i^{-1}(j)}$, that is, in the block $B_{ij}$	there is a non-zero entry at position $(\sigma_i^{-1}(j)\,,\,\sigma_j^{-1}(i))$  and from $(ii)$ this is the unique one.
	
\end{proof}
We now turn to the description of  the representing matrix of the direct product of  two  st-YBE. In the following, we use the following notations:
\begin{gather*}
\hat{i}\,=\,\lceil \frac{i}{m}\rceil,\; \textrm{where}\; \lceil ..\rceil\; \textrm{denotes the ceil function}\\
\bar{i}= i (mod \,m),\; \textrm{the residue modulo}\, m\, \textrm{of} \,i \\
\end{gather*}
Whenever the residue is $0$ it is replaced by $m$.
\begin{lem}\label{lem-prop-e}
	Let $(X,r)$ and $(Y,s)$  be  non-degenerate and involutive st-YBE, with $\mid X\mid=n$, $\mid Y\mid=m$. 	Let  $(Z,t)$ be their direct product, as defined in Definition \ref{defn-new-solution}, with linearisation  $e:W\otimes W\rightarrow W \otimes W$.   Assume that $e$  is partitioned into $(nm)^2$ square blocks $E_{ij}$ of order $nm$. 
	 Let  $1 \leq i\leq n,\, 1 \leq j \leq m$. Then, in the block $E_{ij}$: 
	\begin{itemize}
		\item  we read the images of all the elements $\{T_{\hat{j}}^{\bar{j}}\,\otimes \,T_{p}^{q}\,\mid 1 \leq p  \leq n,1 \leq q \leq m\} $ relative to the subset of basis elements  $\{T_{\hat{i}}^{\bar{i}}\,\otimes \,T_{p}^{q}\,\mid 1 \leq p  \leq n,1 \leq q \leq m\} $ 
	\item there is a unique non-zero entry at position 
	\begin{gather}
(\;(\sigma^{-1}_{\hat{i}}(\hat{j})-1)m \,+\,\alpha_{\bar{i}}^{-1}(\bar{j})\;,\; 
(\sigma^{-1}_{\hat{j}}(\hat{i})-1)m \,+\,\alpha_{\bar{j}}^{-1}(\bar{i})\;)
	\end{gather}
		\end{itemize}
\end{lem}
\begin{proof}
	From Lemma \ref{lem-new--invol+braided},  $(Z,t)$ is a non-degenerate and involutive  st-YBE, so it satisfies  the properties from Lemma  \ref{lem-propblock--c}, that is  there is a unique non-zero entry at each block of $e$.  In the first $nm$ vertical blocks of the form $E_{*1}$, we read the images of 
	all the elements $\{T_{1}^{1}\,\otimes \,T_{p}^{q}\,\mid 1 \leq p  \leq n,1 \leq q \leq m\} $ relative to all the basis elements, in the next $nm$ vertical blocks of the form $E_{*2}$, we read the images of 
	all the elements $\{T_{1}^{2}\,\otimes \,T_{p}^{q}\,\mid 1 \leq p  \leq n,1 \leq q \leq m\} $ relative to all the basis elements and so on. 
	
	In the $nm$ vertical blocks of the form $E_{*\,m+1}$, we read the images of 
	all the elements $\{T_{2}^{1}\,\otimes \,T_{p}^{q}\,\mid 1 \leq p  \leq n,1 \leq q \leq m\} $ relative to all the basis elements,  and 
		in the $nm$ vertical blocks of the form $E_{*\,m+2}$, we read the images of 
	all the elements $\{T_{2}^{2}\,\otimes \,T_{p}^{q}\,\mid 1 \leq p  \leq n,1 \leq q \leq m\} $ relative to all the basis elements,  and so on. 
		In the $nm$ vertical blocks of the form $E_{*\,2m+1}$, we read the images of 
	all the elements $\{T_{3}^{1}\,\otimes \,T_{p}^{q}\,\mid 1 \leq p  \leq n,1 \leq q \leq m\} $ relative to all the basis elements,  and so on.
	Let  $1 \leq i,j \leq n$.  So, from the above,  in the vertical blocks of the form  $E_{*j}$,  we read the images of all the elements $\{T_{\hat{j}}^{\bar{j}}\,\otimes \,T_{p}^{q}\,\mid 1 \leq p  \leq n,1 \leq q \leq m\} $ relative to all the basis elements. Using  the same considerations with the rows instead of the columns,we  find that  in the block $E_{ij}$, we read the images of all the elements $\{T_{\hat{j}}^{\bar{j}}\,\otimes \,T_{p}^{q}\,\mid 1 \leq p  \leq n,1 \leq q \leq m\} $ relative to the subset of basis elements  $\{T_{\hat{i}}^{\bar{i}}\,\otimes \,T_{p}^{q}\,\mid 1 \leq p  \leq n,1 \leq q \leq m\} $.  From Lemma \ref{lem-i-sigma-1},  $e(\; T_{\hat{j}}^{\bar{j}}\,\otimes \,T_{\sigma_{\hat{j}}^{-1}(\hat{i})}^{\alpha_{\bar{j}}^{-1}(\bar{i})}\;)\;=\; 
	T_{\hat{i}}^{\bar{i}}\,\otimes \,T_{\sigma_{\hat{i}}^{-1}(\hat{j
		})}^{\alpha_{\bar{i}}^{-1}(\bar{j})}\;$. So, there is a non-zero entry at the position  	$(\;(\sigma^{-1}_{\hat{i}}(\hat{j})-1)m \,+\,\alpha_{\bar{i}}^{-1}(\bar{j})\;,\; 
	(\sigma^{-1}_{\hat{j}}(\hat{i})-1)m \,+\,\alpha_{\bar{j}}^{-1}(\bar{i})\;)$, and it is the unique one.
\end{proof}

\subsection{Properties of the matrix $c \boxtimes d$}

\begin{lem}\label{lem-prop-tracy}
	Let $(X,r)$ and $(Y,s)$  be  non-degenerate and involutive  st-YBE, with $\mid X\mid=n$, $\mid Y\mid=m$. Let $c:V\otimes V\rightarrow V \otimes V$ and $d:V'\otimes V'\rightarrow V' \otimes V'$ be  the  the linearisations of $r$ and $s$ respectively.  Assume that $c$  is partitioned into $n^2$ square blocks $B_{ij}$ of order $n$ and $d$  is partitioned  into $m^2$ square blocks $B'_{ij}$ of order $m$.  Let $c \boxtimes d$ denote the Tracy-Singh product of $c$ and $d$ with these partitions.  Assume that $c \boxtimes d$  is partitioned into $(nm)^2$  square blocks  of order $mn$, denoted by $\mathcal{B}_{ij}$. 
		 Let  $1 \leq i\leq n, 1 \leq j \leq m$. Then:
		 \begin{enumerate}[(i)]
		 	\item  The block $\mathcal{B}_{ij}$ is obtained from   $\mathcal{B}_{ij}\,=\, B_{\hat{i}\,\hat{j}}\,\otimes B'_{\bar{i}\,\bar{j}}$.
	\item 	 In the block $\mathcal{B}_{ij}$,  there is a unique non-zero entry at position: 
		\begin{gather}
		(\;(\sigma^{-1}_{\hat{i}}(\hat{j})-1)m \,+\,\alpha_{\bar{i}}^{-1}(\bar{j})\;,\; 
		(\sigma^{-1}_{\hat{j}}(\hat{i})-1)m \,+\,\alpha_{\bar{j}}^{-1}(\bar{i})\;)
		\end{gather}
		 \end{enumerate}
\end{lem}

\begin{proof}
	$(i)$ results from the definition of  the Tracy-Singh product of $c$ and $d$ with these partitions, where at  each row of $d$ there are $m$ blocks. We refer to Equation (\ref{eqn-tracy-blocks}) for an illustration.\\
$(ii)$ 	From Lemma \ref{lem-propblock--c}$(iii)$, there is a  unique non-zero entry in the block $B_{\hat{i}\,\hat{j}}$ in $c$ 
at position  $(\;\sigma^{-1}_{\hat{i}}(\hat{j})\,,\,\sigma^{-1}_{\hat{j}}(\hat{i})\;)$ and  there is a  unique non-zero entry in the block $B'_{\bar{i}\,\bar{j}}$ in $d$  at position  $(\;\alpha^{-1}_{\bar{i}}(\bar{j})\,,\,\alpha^{-1}_{\bar{j}}(\bar{i})\;)$. In Figure \ref{fig-comput-2-blocks}, we illustrate   the computation of $\mathcal{B}_{ij}$:
\begin{figure}[H]
	$\mathcal{B}_{ij}=B_{\hat{i}\,\hat{j}}\,\otimes B'_{\bar{i}\,\bar{j}}=\begin{pmatrix}
	0 & 0 &0 & ...&	1\\
	0 &0 & 0 &...&0 \\
	.. &.. & .. &...&.. \\
	0 &0 & 0 &...&0 \\
	\end{pmatrix} \otimes \begin{pmatrix}
0 & 0 &0 & ...&	0\\
0 &0 & 0 &...&0 \\
.. &.. & .. &...&.. \\
1 &0 & 0 &...&0 \\
	\end{pmatrix}=	 \begin{pmatrix}
		\mat{0}_{m} & 	\mat{0}_{m} &	\mat{0}_{m} & ...&	B'_{\bar{i}\,\bar{j}}\\
		\mat{0}_{m} &	\mat{0}_{m} & 	\mat{0}_{m} &...&	\mat{0}_{m} \\
	.. &.. & .. &...&.. \\
		\mat{0}_{m} &	\mat{0}_{m} & 	\mat{0}_{m} &...&	\mat{0}_{m} \\
	\end{pmatrix}$.
\caption{$\mathcal{B}_{ij}=B_{\hat{i}\,\hat{j}}\,\otimes B'_{\bar{i}\,\bar{j}}$, where $B_{\hat{i}\,\hat{j}}$ and $ B'_{\bar{i}\,\bar{j}}$   of order $n$ and $m$ resp. }\label{fig-comput-2-blocks}
\end{figure}

From the definition of $\otimes$,   in  the block $\mathcal{B}_{ij}$,  there are $n^2-1$ zero square  sublocks of order $m$ and an unique non-zero square  sublock of order $m$ at position $(\;\sigma^{-1}_{\hat{i}}(\hat{j})\,,\,\sigma^{-1}_{\hat{j}}(\hat{i})\;)$, since there is a unique non-zero entry in $B_{\hat{i}\,\hat{j}}$ at that position.  In this non-zero sublock, there is a unique non-zero entry at position $(\;\alpha^{-1}_{\bar{i}}(\bar{j})\,,\,\alpha^{-1}_{\bar{j}}(\bar{i})\;)$ (in the sublock), since 
there is a unique non-zero entry in $B'_{\bar{i}\,\bar{j}}$   at that position. So, there is a unique non-zero entry  in the block $\mathcal{B}_{ij}$,  and it occurs at position
\begin{gather}
(\;(\sigma^{-1}_{\hat{i}}(\hat{j})-1)m \,+\,\alpha_{\bar{i}}^{-1}(\bar{j})\;,\; 
(\sigma^{-1}_{\hat{j}}(\hat{i})-1)m \,+\,\alpha_{\bar{j}}^{-1}(\bar{i})\;)
\end{gather}
\end{proof}
We reformulate Theorem \ref{thmA} in a more precise way and give its proof.
\begin{thm*} \textbf{1}
	Let $(X,r)$ and $(Y,s)$  be  non-degenerate and involutive  st-YBE, with $\mid X\mid=n$, $\mid Y\mid=m$. Let $c:V\otimes V\rightarrow V \otimes V$ and $d:V'\otimes V'\rightarrow V' \otimes V'$ be  the linearisations of $r$ and $s$  respectively.  Assume that $c$  is partitioned into $n^2$ square blocks $B_{ij}$ of order $n$ and $d$  is partitioned  into $m^2$ square blocks $B'_{ij}$ of order $m$.  	Let $c \boxtimes d$ denote the Tracy-Singh product of $c$ and $d$ with these partitions.   Let  $(Z,t)$ be their direct product, as defined in Definition \ref{defn-new-solution}, with linearisation  $e:W\otimes W\rightarrow W \otimes W$. Then $c \boxtimes d$ is the representing matrix of  the linear automorphism $e$  with respect to the basis $Z\otimes Z$.
\end{thm*}
\begin{proof}
 Assume that $c \boxtimes d$  and $e$ are partitioned into $(nm)^2$ square blocks of order $mn$, denoted by $\mathcal{B}_{ij}$ and $E_{ij}$ respectively, $1 \leq i \leq n$, $1 \leq j \leq m$. 
 Let $1 \leq i \leq n$, $1 \leq j \leq m$.  From Lemmas \ref{lem-prop-tracy} and \ref{lem-prop-e}, there is a unique non-zero entry in $\mathcal{B}_{ij}$ and $E_{ij}$, and it occurs at the same position in the blocks. So, as all the blocks are equal, the matrices $e$ and $c \boxtimes d$ are equal, that is 
 $c \boxtimes d$ is the representing matrix of  the linear automorphism $e$  with respect to the basis $Z\otimes Z$.
\end{proof}


\bigskip\bigskip\noindent
{ Fabienne Chouraqui}

\smallskip\noindent
University of Haifa at Oranim, Israel.

\smallskip\noindent
E-mail: {\tt fabienne.chouraqui@gmail.com} \\

                {\tt fchoura@sci.haifa.ac.il}
\end{document}